\documentclass{article}

\usepackage{amssymb}
\usepackage{amsmath}
\usepackage{amsthm}

\newtheorem{theorem}{Theorem}
\newtheorem{lemma}{Lemma}[section]
\newtheorem{propo}{Proposition}[section]
\newtheorem{claim}{Claim}[section]
\theoremstyle{definition}

\newtheorem{remark}{Remark}[section]

\newcommand{\ep}{\varepsilon}

\newcommand{\al}{\alpha}

\newcommand{\e}{\ep}

\newcommand{\N}{\mathbb{N}}
\newcommand{\R}{\mathbb{R}}

\newcommand{\B}{\textbf{B}}

\newcommand{\cG}{\mathcal{G}}

\newcommand{\BGrd}[1][d]{\textrm{\textbf{BGr}}_{#1}}

\newcommand{\OI}{\{0,1\}}
\newcommand{\defeq}{\stackrel{\mbox{\scriptsize def}}{=}}

\title{Borel oracles. An analytical approach to constant-time algorithms.
\footnote{AMS
Subject Classification: Primary 68R10, Secondary 03E15 ,
\, Research sponsored by OTKA Grant No. 67867, No. 69062}}

\author{G\'abor Elek, G\'abor Lippner}

\begin{document}
\maketitle

\begin{abstract} In \cite{NO},
Nguyen and Onak constructed the first 
constant-time algorithm
for the approximation of the size of the
maximum matching in bounded degree graphs.
The Borel oracle machinery is a tool that can be used
to convert some statements in
Borel graph theory to theorems in the field of constant-time  
algorithms. 
In this paper we illustrate the power of this tool to prove the existence of 
the above mentioned
constant-time approximation algorithm.
\end{abstract}

\section{Introduction}
\subsection{Borel graphs}
Borel graphs were introduced by Kechris, Solecki and Todorcevic in
\cite{KST}.  Let $X$ be a standard Borel space, say the $[0,1]$
interval, with its Borel structure. Let $E\subset X^2$ be a Borel graph
relation, that is a Borel subset, such that $(x,y)\in E$ implies
$(y,x)\in E$ and $(x,x)\notin E$ for any $x\in X$. The $\cG=(X,E)$
system is called a Borel graph. In this paper we are interested in
Borel graphs in which all the degrees are bounded. Kechris, Solecki
and Todorcevic proved that in such graphs there exists a maximal
independent Borel subset and using that they proved that if the
degrees are bounded by $d$ then there exists a Borel coloring $c:X\to
\{1,2,\dots,d+1\}$ of the vertices. In order to appreciate this result
let us note that the analogue of Vizing's theorem, the existence of a
Borel edge-coloring $c:E \to \{1,2,\dots,d+1\}$ is still
unknown. Also, there exist some Borel trees with countable degrees
such that they do not have countable Borel colorings (of course they
have non-Borel colorings with only two colors). 

\noindent
A Borel matching in $\cG(X,E)$ is a Borel set $M\subset E$ such that all of its
components are single edges. For a matching $M$ we shall denote by
$X_M$ the Borel subset of $X$ that consists of those points that are
matched by $M$. An augmenting path for $M$ is a finite 
path $x_0,x_1,\dots, x_{2k+1}$ such that
\begin{itemize}
\item $(x_{2j}x_{2j+1}) \notin M$ and $x_{2j+1}x_{2j+2} \in M$ for all $0\leq
  j \leq k$.
\item The two endpoints $x_0,x_{2k+1} \notin X_M$.
\end{itemize}
We shall say that the ``length'' of an augmenting path $x_0,x_1,\dots,
x_{2k+1}$ is $k$, and we shall often refer to it as a $k$-augmenting
path.  One can also define a (one ended) infinite augmenting path in
the same way.  If $M$ is a matching and $P$ is the edge-set of an
augmenting path, then $M' = M \triangle P$ is also a matching, in fact
it is larger in the sense that $X_{M} \subsetneq X_{M'}$. The matching $M'$
will often be refered to as $M$ 
improved by $P$.


Our result is about the existence of Borel matchings without short 
augmenting paths.
\begin{propo}~\label{borelthm}
Let $T > 0$. Then any bounded-degree Borel graph
contains a Borel matching $M$ that has no $k$-augmenting paths for
any $k \leq T$.
\end{propo}

One might think that the previous proposition implies the existence of 
Borel matchings
without {\it any} finite augmenting paths. Nevertheless, we have a 
counterexample.
\begin{propo}\label{borelthm2} There exists a Borel graph $\cG$ such that all
the Borel matchings of $\cG$ contain augmenting paths of finite length.
\end{propo}
\subsection{Constant-time approximation algorithms}
In practice, the most efficient algorithms are linear algorithms, that is,
 those that take time 
proportional to their input. However, in some cases, even to 
access and store the total input in real time is impossible. 
These technological problems lead to the development of constant-time
algorithms that use only a small (random) fraction of the input and still 
efficiently analyze massive data-sets (see \cite{Rub} for a recent survey).
First, let us explain what we mean by {\it constant-time
algorithms} for bounded-degree graphs.

\noindent
Fix a constant $d>0$ and denote the set of all finite simple graphs
with vertex degree bound $d$ by $Graph_d$. Our goal is to construct a
randomized algorithm that calculates the size of the maximum matching
of graph $G\in Graph_d$ with high precision and high probability in
constant time (that is the running time of
our algorithms should be independent on the
size of the graphs).  In general, let $P:Graph_d\to \R$ be a graph
parameter (e.g. the ratio of the maximal matching, maximal independent
set, minimal vertex cover). A constant time approximation of the
parameter means that for any fixed $\e>0$ we have a tester that takes
a graph $G$ as an input, then
\begin{itemize}
\item It gives random uniform real label to the vertices from the interval
  $[0,1]$.
\item Then the tester randomly picks $C(\e)$ vertices of the graph and explores
the $C(\e)$-neighbourhood of the chosen vertices.
\item Based on the sampling the tester gives an estimate $P_E(G)$ for
the parameter $P(G)$ of the maximum matching  such a way that
$$Prob (|P(G)-P_E(G)|>\e) <\e\,.$$
\end{itemize}
Nguyen and Onak constructed the first constant-time approximation
algorithm for the ratio of the maximum matching that is $m(G) =
\frac{|M(G)|}{|V(G)|}$ where $M(G)$ is a maximum matching.  Note that
such approximation algorithm do not exists for the ratio of the
maximal independent set \cite{Bol} or the minimal vertex cover
\cite{BOT}.

\noindent
How can we see that such testers exist for a certain graph parameter ?
The existence of constant-time algorithms can be translated to a topological
statement using the notion of the Benjamini-Schramm graph convergence
\cite{BS}. 
Let us briefly recall the definition of graph convergence.
A rooted $(r,d)$-ball is a finite, simple, connected graph $H$ such that
\begin{itemize}
\item $\mbox{deg}(y)\leq d$ if $y\in V(H)$\,.
\item $H$ has a distinguished vertex $x$ (the root).
\item $d_G(x,y)\leq r$ for any $y\in V(H)$.
\end{itemize}
For $r\geq 1$, we denote by $U^{r,d}$ the finite set of rooted isomorphism
classes of rooted $(r,d)$-balls. Let $G(V,E)$ be a finite graph with vertex
degree bound $d$. For $\alpha\in U^{r,d}$, $T(G,\alpha)$ denotes the
set of vertices $x\in V(G)$ such that there exists a rooted isomorphism
between $\alpha$ and the rooted $r$-ball $B_r(x)$ around $x$.
 Set $p_G(\alpha):=\frac{|T(G,\alpha)|}
{|V(G)|}\,.$ Thus we associated to $G$ a probability distribution on $U^{r,d}$
for any $r\geq 1$. Let $\{G_n\}^\infty_{n=1}\subset Graph_d$ be a 
sequence of finite
simple graphs such that $\lim_{n\to\infty}
 |V(G_n)|=\infty$. Then $\{G_n\}^\infty_{n=1}$ 
is called
{\it convergent} if for any $r\geq 1$ and $\alpha\in U^{r,d}$,
$\lim_{n\to\infty} p_{G_n}(\alpha)$ exists.
If $P$ is a graph parameter and $\lim_{n\to\infty} P(G_n)$ exists for
any convergent graph sequence then there exists a constant-time algorithm
for the approximation of $P$ (Theorem 3, \cite{ElekRSA}).

\noindent
Our main goal is to prove the following result.
\begin{theorem}~\label{convthm} 
Let $\{G_n\}^\infty_{n=1}\subset Graph_d$ be a convergent graph sequence
 in the sense of Benjamini-Schramm. Then $\lim_{n\to\infty} m(G_n)$ exists.
\end{theorem}
By the previous remark, Theorem \ref{convthm} implies that there exists
a constant-time algorithm for approximating the maximum matching. Note
that our result is highly non-constructive. In \cite{NO} an effective upper
bound is given for the running time in terms of $d$ and $\e$.
To prove the theorem we use Proposition \ref{borelthm} via  the 
{\it Borel oracle machinery}. The
 machinery can be described informally the following way.
 There is one single Borel graph with one given matching
 without short augmenting paths (The Borel Oracle). Sampling a finite 
graph can be
viewed as ``asking'' the oracle about the optimal choice of picking a matching 
in this particular graph.
\vskip 0.2in
\noindent
{\bf Acknowledgement:} We would like to thank Michael Krivelevich for calling
our attention to \cite{NO}.
\section{The proof of Propositions \ref{borelthm} and \ref{borelthm2}}
Let us start with a simple lemma, which is a variation of the classical
observation that a matching in a finite graph without augmenting paths is
always of maximum size (see also Lemma 6 in \cite{NO}).
\begin{lemma}~\label{majdnemmax}
Let $G$ be a finite graph. Let $M$ be a matching such that there are
at most $q\cdot |G|$ vertices from which a
augmenting path shorter than $T$ starts and let $M_m$ be a maximum size
matching of $G$. Then 
\[ \frac{|M|}{|G|} \leq \frac{|M_m|}{|G|} \leq \frac{|M|}{|G|} \cdot
\frac{T+1}{T} + q.\]
\end{lemma}

\begin{proof} Let $Q \subset G$ denote those vertices from which 
a short augmenting path starts. Consider the symmetric difference of the
edge set of $M$ and $M_m$. This clearly consist of disjoint paths and
cycles in which the edges coming from $M$ and $M_m$ are
alternating. Hence cycles have to be of even length. In each cycle and
in each path of even length there are just as many edges coming from
$M$ as from $M_m$. On the other hand, if $x_0,x_1,\dots,x_{2k+1}$ is a
path of odd length then $x_0x_1$ must come from $M$ (and then
$x_{2k}x_{2k+1}$ also), otherwise it would be an augmenting path for
$M_m$ that is clearly impossible. But then it must be an augmenting
path for $M$ and so either $x_0 \in Q$ or $k \geq T$ by our
assumption. In the first case there are at least $T$ edges of
$M$ in the path. Thus we have
\[ |M_m| - |M| = \mbox{number of odd paths in $M \triangle M_m$} \leq
|Q| + \frac{|M|}{T}\,.\] Hence the lemma follows.
\end{proof}

We start the proof of Proposition \ref{borelthm} by analysing the case 
of one single
countable, connected graph. The construction for this case will be
used in the case of Borel graphs too.

\begin{lemma}
Let $G(V,E)$ be a countable, connected graph with degree bound $d$, and let
$T > 0$. Then $G$ contains a matching $M$ that has no $k$-augmenting paths 
for any $k \leq T$.
\end{lemma}

\begin{proof}
Let us enumerate the vertices of $G$ such that each vertex is listed
infinitely many times. Let $v_1,v_2,\dots$ be this enumeration.  We
construct $M$ as a limit of a sequence of matchings
$M_0,M_1,M_2,\dots$. We start with the empty matching $M_0$. In the
$i$-th step we take $M_{i-1}$. We look at the $4T$-neighborhood
$B_{4T}(v_i)$ of $v_i$ in $G$. This is a finite subgraph of size at most
$(d+1)^{4T}$. We look for $k$-augmenting paths
for $M_{i-1}$ of length at most $k \leq T$ in this subgraph, and
improve $M_{i-1}$ every time we find one. Since each improvement
increases the number of edges of $M_{i-1}$ in $B_{4T}(v_i)$, this
process terminates in a finite number of steps. Let $M_i$ denote the
improved matching.

\begin{claim} The sequence $M_i$ stabilizes in the sense that for each
  edge $(xy) \in E$ there is an integer $N$ such that $(xy) \in M_n$
  for all $n \geq N$ or $(xy) \notin M_n$ for all $n \geq N$. Hence we
  can define $M$ to be the (edge-wise) limit of $M_i$.
\end{claim}

To see this it is enough to show that the ``status'' of a fixed edge
can change only finitely many times during the process. Each change of
status involves improving by an augmenting path of length at most
$T$. The endpoints of such a path necessarily lie in $B_{2T+1}(x)$. At
each improvement, the number of matched vertices in this neighborhood
strictly increases. Since the neighborhood is finite, there can be
only finitely many improvements effecting the edge $(xy)$.

\begin{claim} $M$ is a matching without any $k$-augmenting paths for
  any $k \leq T$.
\end{claim}

Suppose there is a $k$-augmenting path $x_0,\dots,x_{2k+1}$ for
$M$. By the precious claim we can choose $N$ so large that for any $n
\geq N$ we have $M|_{B_{4T}(x_0)} = M_n|_{B_{4T}(x_0)}$. But there
will also be an $n \geq N$ for which $v_n = x_0$. Then in the $n$-th
step the path $x_0,\dots,x_{2k+1}$ has to be a $k$-augmenting path for
$M_n$ in the $4T$-neighborhood of $v_n$, so $M_n|_{B_{4T}(x_0)} \neq
M_{n+1}|_{B_{4T}(x_0)}$ which is clearly a contradiction.
A similar (but much simpler) argument shows that $M$ is in fact a matching.
\end{proof}

\begin{remark}
Note that instead of improving one neighborhood in each step,
it is perfectly fine if in each step we improve many (even infinitely
many) disjoint $4T$-neighborhoods.
\end{remark}

\begin{proof}[Proof of Proposition~\ref{borelthm}]
Let us color (in a Borel way) the vertices of our graph $\cG$ with 
finitely many $(K)$ colors
such that any two vertices of the same color have disjoint
$4T$-neighborhoods. This can be done by the result of Kechris, Solecki and
Todorcevic. Indeed, let $\cG'$ be the Borel graph that we obtain
if we connect two vertices in $\cG$ when their distances are not greater than
$4T$ and consider a Borel coloring of $\cG'$ by finitely many colors.

Now we execute the algorithm described in the previous lemma in the
following way: Let $M_0$ be the empty matching. In step $n$ we take
the matching $M_n$ constructed so far, and take the $n$-th color class
modulo $K$. We improve the matching $M_n$ in the $4T$-neighborhood of
each vertex of the color class parallel just as in the lemma. Hence we
obtain a sequence of matchings $M_1,M_2,\dots$. Each $M_n$ will be
Borel, since one step can be carried out in a Borel way: there are
only finitely many different $4T$-neighborhoods together with a
matching. On each such neighborhood-with-matching we can fix a unique
way of improving the matching, and hence each local improvement depends
only on the starting state of the neighborhood of the vertex, which on
the other hand depends on the vertex in a Borel way. There could be
one more issue: the neighborhoods might have symmetries and then
the choice of improvement might not be unique. However, using the colorings
of the vertices we can order the augmented paths in the neighborhoods and
always improve the first augmented paths to break the tie in case of
a symmetry.

Now restricting this sequence to any single connected component of $G$
we get a sequence of matchings that we constructed in the lemma. Hence
the sequence stabilizes and the limit $M$ will be a Borel matching
without any $k$-augmenting paths for any $k \leq T$.
\end{proof}

\vskip 0.2in
Now we prove Proposition \ref{borelthm2}.
Recall that Laczkovich \cite{Lac} constructed a $2$-regular Borel graph 
$\cG(X,E)$
on $X=[0,1]$, such that for any matching $M\subset E$, $\mu(X_M)<1$, where
$\mu$ denotes the Lebesgue measure.
Assume now that $M\subset E$ is a matching without finite
augmenting paths, and let $X_0 = X \setminus X_M$. If $x_0 \in X_0$
and $x_0,x_1,x_2,\dots$ is a path in $\cG$ then $(x_1x_2)$ has to be in
$M$ (since that graph is 2-regular and $x_0x_1$ is not a 0-augmenting
path), then $x_3x_4$ has to be in $M$ for a similar reason, and so
on. It follows that $x_1,x_2,\dots \in X_M$. Also if $x_0 \neq y_0 \in
X_0$ and $y_0,y_1,y_2,\dots$ is a path in $\cG$ then $y_1,y_2,\dots \in
X_M$ and the $x_i$'s are all different from the $y_j$'s otherwise
there would be again a finite augmenting path. Therefore, from every $x \in
X_0$ an infinite path in $X_M$ starts and these paths are all
disjoint. In fact, from each $x$ there are two such paths starting
with either edge incident to $x$. Let us choose the one which starts
with the smaller vertex. Thus we have Borel maps $T^i:X_0\to [0,1]$ such that
if $x_0\in X_0$ then $T^i(x_0)=x_i$. By our construction, 
$\mu(T^i(X_0))=\mu(X_0)$ and $T^i(X_0)\cap T^j(X_0)=\emptyset$ if $i\neq j$. 
Hence $\mu(X_0) = 0$, leading to a contradiction. \qed

\section{Invariant measures on 
Borel graphs as the limits of finite graphs}~\label{bgraphsec}

 In \cite{BS}, Benjamini and Schramm constructed unimodular measures
on the space of connected, countable graphs as the limit objects
of convergent graph sequences. In this paper we consider invariant
measures
on Borel graphs as limit objects. The two notions are closely related (see
Example 9.9, \cite{AL}). 

\noindent
First of all, let us recall the notion of invariant measures on Borel graphs 
from \cite{Kech}. Let $\cG(X,E)$ be a Borel graph. A Borel involution $T:X\to
X$ defines a matching of $\cG$ if for any $x\in X$, $x\neq T(x)$: $x$ is
connected to $T$ in $\cG$. A probability measure on $X$ is $\cG$-invariant if
it is preserved by all Borel involutions $T$ that define a matching in
$\cG$. Note that for the Borel graph in the proof of Proposition
\ref{borelthm2} the Lebesgue measure is an invariant probability measure.

\noindent
Now let us recall the limit object construction from \cite{ElekRSA}. 
 Let $B = \OI^{\N}$ be the Bernoulli space of
  0-1-sequences with the standard product measure $\nu$. A (rooted)
  $B$-graph is a (rooted) graph $G$ and a function $\tau_G : V(G) \to
  B$. Two rooted $B$-graphs $G$ and $H$ are said to be isomorphic if
  there exists a rooted isomorphism $\psi : V(G) \to V(H)$ such that
  $\tau_H(\psi(x)) = \tau_G(x)$ for every $x \in V(G)$. The set of
  isomorphism classes of all countable rooted $B$-graphs with degree
  bound $d$ is denoted by $\BGrd$. In fact it will be more convenient
  to use the subset of this in which the labels are required to be
  different for every vertex in the graph. This subset shall be
  denoted $\B_d$.

Let $U^r = U^{r,d}$ denote the set of isomorphism classes of rooted
$r$-balls with degree bound $d$ and vertices labeled with $\OI^r$. For
a $B$-graph $BG$ and a vertex $x\in V(BG)$ by $B_r(x)\in U^r$
we shall denote the rooted $r$-ball around $x$ with the labels
truncated to the first $r$ digits.  For any $\al \in U^r$ and a
$B$-graph $BG$ we define the set $T(BG,\alpha) \defeq \{x \in V(G):
B_r(x) \cong_B \alpha\}$ and define $p_{BG}(\alpha) \defeq
\frac{|T(BG,\alpha)|}{|V(G)|}$.  For $\al \in U^r$ let us define
$T(\BGrd,\al) = \{x \in \BGrd : B_r(x) \cong \al\}$. There is a
natural metric on $\BGrd$. If $X,Y\in \BGrd$ then
\[d_b(X,Y)=2^{-r}\,,\]
where $r$ is the maximal number such that $B^r(x)\cong B^r(y)$,
where $x$ is the root of $X$, $y$ is the root of $Y$.  The subsets
$T(\BGrd,\al): \al \in U^r, r \in \N$ are closed-open sets and
generate the Borel-structure of $\BGrd$.

Let $\{BG_n\}^\infty_{n=1}$ be a sequence of $B$-graphs.  We say that
$\{BG_n\}^\infty_{n=1}$ converges if for any $\al \in U^r$,
$\lim_{n\to\infty} p_{BG_n}(\alpha)= \mu(T(\BGrd,\al))$ exists. In
this case $\mu$ naturally extends to a Borel-measure on $\BGrd$. We
call $\mu$ the limit measure of $\{BG_n\}^\infty_{n=1}$.

$\BGrd$ can be given a Borel graphing structure $\cG$ in a natural
way: two rooted $B$-graphs $BG$ and $BH$ with roots $x \in BG, y \in
BH$ are adjacent if there is a vertex $z \in BG$ adjacent to $x$ such
that $BG$ with root $z$ is isomorphic to $BH$ with root $y$. Obviously
this is a Borel graph with degree bound $d$. The following proposition is the
straightforward consequence of Proposition 2.2 and Corollary 3.1 of
\cite{ElekRSA}( see also Example 9.9 of \cite{AL}).
\begin{propo} \label{enlemmam}
Let $\{G_n\}^\infty_{n=1}$ be a convergent graph sequence.  Let
$\{BG_n\}^\infty_{n=1}$ be a uniformly random $B$-labelling of the
vertices of $G_n$. Then $\{BG_n\}^\infty_{n=1}$ almost surely
converges to a measure $\mu$ which is concentrated entirely on
$\B_d$, and restricted to this subset it is $\cG$-invariant.
\end{propo}

\section{Proof of Theorem~\ref{convthm}}

In this section $\cG$ denotes the the natural Borel graph on $\B_d$.
Let $\{G_n\}^\infty_{n=1}$ be a convergent graph sequence. Then by
Proposition~\ref{enlemmam} there is a $B$-labeling of the vertices so
that the resulting $B$-graph sequence $\{BG_n\}^\infty_{n=1}$ converges to a
$\cG$-invariant measure $\mu$ on $\B_d$. 

\noindent
Let us fix an $\ep > 0$ and choose an integer $T > 1/\ep$. By
Proposition~\ref{borelthm} there exists a Borel matching $M$ of the
graphing $\cG$ that has no $t$-augmenting paths for any $t \leq T$. 

\noindent
This matching $M$ is our Borel oracle.
Using $M$ we will be able to construct matchings $M_n$ of $G_n$ that
almost have the same property. Any point $x \in \B_d$ represents a
labeled, rooted graph. Let $l(x)$ denote the label of the root of this
graph. If the neighbors of $x$ are $y_1,y_2,\dots,y_s$ then
$l(y_1),\dots,l(y_s)$ are exactly the labels of the neighbors of the
root of $x$. Since all the labels of the graph are different, we can
order $y_1,\dots,y_s$ according to their $l$-values. If $l(y_1) <
l(y_2) < \dots < l(y_s)$ then $y_i$ will be called the $i$-th neighbor
of $x$. Let $X_0 = \B_d \setminus (\B_d)_M$ and let 
\begin{multline*}X_{ij} = \{x \in (\B_d)_M : \mbox{ $x$ is connected
    to its $i$-th neighbor in the} \\ \mbox{matching $M$ and $x$ is
  the $j$-th neighbor of its own pair}\}.
\end{multline*}

Therefore if $xy \in M$ and $x \in X_{ij}$ then $y \in X_{ji}$. This way we
get a Borel partition of $\B_d$ into $d^2+1$ sets. It is easy to see
that this partition can be approximated by another one (denoted by
$X_{ij}'$) in which each part is a finite union of closed-open
sets, and that the total symmetric difference $H = \cup_{1\leq i,j \leq d}
 (X_{ij}
\triangle X_{ij}') \cup (X_0 \triangle X_0')$ is so small that $\mu(H)
\leq \frac{\ep}{(d+1)^{2T+1}}$. 
Since the approximating sets are finite unions of
closed-open sets this means that there is an $r$ such that for any $x$
the neighborhood $B_r(x)$ determines which $X_{ij}'$ contains
$x$. Since we are in $\B_d$, by choosing $r$ large enough we may
assume that we only use such neighborhoods in which all the $r$-digit
labels are different. (More precisely those neighborhood types for
which this does not hold can be put into $H$.)

Now let us look at a vertex $v \in BG_n$, and look at its $B_r(v)$
neighborhood. Let us find the unique $X_{ij}'$ (or $X_0'$) that
contains the points having this neighborhood.  If it is in $X_{ij}'$
then we look at its $i$-th neighbor $w$, and look at $B_r(w)$. If
$B_r(w)$ is contained in $X_{ji}'$ and $v$ is indeed the $j$-th
neighbor of $w$ then we match $v$ to $w$. Otherwise we don't match $v$
at all. This way we obtain a matching $M_n$ of $BG_n$ (and of $G_n$ of
course). Indeed, since $\{X'_{ij}\}\cup X'_0$ is a partition the
degrees of any vertex in the subgraph $M_n$ can be either $0$ or $1$.

We claim that there cannot be too many vertices that are endpoints of
a $t$-augmenting path in the matching $M_n$ if $t \leq T$. Let
$v_0,v_1,\dots,v_{2t+1}$ be a $t$-augmenting path in $BG_n$. Let $\al
\in U^{r+2t+1}$ denote the isomorphism type of
$B_{r+2t+1}(v_0)$.

\begin{claim}
If $x_0 \in \B_d$ and
$B_{r+2t+1}(x_0) \equiv \al$, then $x_0\in H^{2t+1}$ where $H^{2t+1}$ 
denotes the
set of points that are at most $2t+1$ steps from $H$ in the Borel graph
$\cG$. 
\end{claim}
Indeed, lift the path
$v_0,v_1,\dots,v_{2t+1}$ in $BG_n$ to a path $x_0,x_1,\dots,x_{2t+1}$
in $\B_d$ by always finding the appropriate neighbor of the vertices
according to the ordering respectively. If no point of the lifted path
falls into $H$, then by the construction the matching $M$ looks
precisely the same along this path, as $M_n$ looked like along the
$v$-path in $BG_n$. However this is not possible, as $M$ does not
contain a $t$-augmenting path. Hence one of the $x_i$'s has to be in
$H$ or in other words $x_0 \in H^{2t+1}$, proving the claim.

\vskip 0.1in
\noindent
We will need the following lemma.
\begin{lemma}
$\mu(H^{2T+1})\leq  (d+1)^{2T+1}\cdot \mu(H)\,.$
\end{lemma}
\proof
It is enough to prove that $\mu(H^1)\leq (d+1) \mu(H)$. First, let us
consider a Borel coloring of the edges by the natural numbers (see
\cite{Kech}). The set $H_n\subset H$ is defined as follows.
A point $p\in H$ is in $H_n$ if and only if there is a point $q\notin H$
such that $(p,q)\in E$ and the color of $(p,q)$ is $n$. The set of
all such points $q$ will be denoted by $K_n$. By the invariance of $\mu$,
$\mu(H_n)=\mu(K_n)$. 
For a $k$-tuple $Q=\{i_1,i_2,\dots,i_k\}$, where $0\leq k\leq d$ and
$i_1< i_2<\dots < i_k$, let $H_Q$ be defined as the points in $H$
that are exactly in the sets $H_{i_1}, H_{i_2}, \dots, H_{i_k}$.
Clearly, $\cup_Q H_Q=H.$

\noindent
Now let $K^{i_j}_Q$ be the set of points outside $H$ that are connected
to a point in $H_Q$ by an edge colored by $i_j$.
By the invariance of $\mu$, $\mu(\cup_{i_j\in Q}K^{i_j}_Q)\leq d\mu(H_Q)$.
Since
$H^1= H\cup (\bigcup_Q\bigcup_{i_j\in Q} K^{i_j}_Q)$, our lemma follows.
\qed

\vskip 0.2in
\noindent
We call $v \in BG_n$ a bad point if there is a $t$-augmenting path (for
any $t \leq T$)
starting from $v$ and denote by $Q_n \subset BG_n$ the set of bad
points and let $q_n = |Q_n|/|V(G_n)|$.  Let $A \subset
U^{r+2T+1}$ denote the set of neighborhood-types of this size
of all the bad points in all the $BG_n$'s. Then by the previous claim
$\cup_{\al \in A} T(\B_d,\al) \subset H^{2T+1}$.
Hence we have 
\begin{multline*} \limsup q_n \leq \limsup \sum_{\al \in A}
  p_{BG_n}(\al) = \sum_{\al \in A} \lim p_{BG_n}(\al) = \\ =
  \mu(\cup_{\al \in A} T(\B_d,\al)) \leq \mu(H^{2T+1}) \leq
  (d+1)^{2T+1}\cdot \mu(H) \leq \ep 
\end{multline*}

Let us introduce the notations $s_n := |M_n| / |V(G_n)|$ and $m_n := m(G_n)$. 
\begin{claim} $\lim_{n\to\infty} s_n=S$ exists.
\end{claim}
Indeed, by definition being a vertex in the matching depends only the
$r+1$-neighborhood type of the given vertex.

\vskip 0.2in
\noindent
By Lemma~\ref{majdnemmax} we have 
$s_n \leq m_n \leq s_n \cdot \frac{T+1}{T} + q_n$. Hence 
\begin{multline*} S=\liminf s_n \leq \liminf m_n \leq 
\limsup
m_n \leq \\ \leq \frac{T+1}{T} \limsup s_n + \limsup q_n \leq \frac{T+1}{T}
S+ \ep \end{multline*} and thus
$$\limsup m_n - \liminf m_n \leq \frac{1}{T} S+\epsilon\leq 3\ep\,.$$

Since $\ep$ can be arbitrarily small, it follows that the sequence
$\{m(G_n)\}^\infty_{n=1}$ is convergent.
\qed

\end{document}